\documentclass[10pt,a4paper]{article}

\usepackage{amssymb,amsmath,amsthm}
\usepackage{graphicx,graphics}
\usepackage{subfigure}
\usepackage{listings}
\usepackage{algorithm}
\usepackage{algorithmic}
\usepackage{tabularx}
\usepackage{appendix}
\newtheorem{thm}{Theorem}

\newtheorem{lma}{Lemma}

\theoremstyle{remark}
\newtheorem{rem}{Remark}

\numberwithin{equation}{section}

\def\mbf#1{\mathchoice{\hbox{\boldmath $\displaystyle #1$}}
       {\hbox{\boldmath $\textstyle #1$}}
       {\hbox{\boldmath $\scriptstyle #1$}}
       {\hbox{\boldmath $\scriptscriptstyle #1$}}}

\newcommand{\tn}{\mspace{0.5mu} |\mspace{-1.2mu}|\mspace{-1.2mu}|\mspace{0.5mu}}
\newcommand{\Keywords}[1]{\par\noindent{\small{\em Keywords\/}: #1}}

\algsetup{indent=2em}

\date{June 19, 2012}

\begin{document}
\title{A Posteriori Error Analysis of Component Mode Synthesis for the
  Frequency Response Problem} \author{ H{\aa}kan Jakobsson
  \thanks{Department of Mathematics and
    Mathematical Statistics, Ume{\aa} University, SE-901 87 Ume{\aa},
    Sweden, email: hakan.jakobsson@math.umu.se } \and Mats G. Larson \thanks{Department of Mathematics and Mathematical Statistics,
    Ume{\aa} University, SE-901 87 Ume{\aa}, Sweden, email: mats.larson@math.umu.se }}

\maketitle

\begin{abstract}
  We consider the frequency response problem and derive a posteriori
  error estimates for the discrete error in a reduced finite element
  model obtained using the component mode synthesis (CMS) method. We
  provide estimates in a linear quantity of interest and the energy
  norm. The estimates reflect to what degree each CMS subspace
  influence the overall error in the reduced solution. This enables
  automatic error control through adaptive algorithms that determine
  suitable dimensions of each subspace.  We illustrate the theoretical
  results by including several numerical examples.
\end{abstract}


\section{Introduction}
Due to the large scale of finite element models of complex structures,
it may be necessary to use reduced finite element models with much
fewer degrees of freedom when performing frequency response analysis
of a structure over a large range of frequencies. Having control over
the reduction error in the approximation is then highly important. In
this paper we derive a posteriori error estimates for the discrete
error in the reduced solution to the frequency response problem
obtained using component mode synthesis (CMS) \cite{craig-bampton1968,
  craig1977chang, hurty1965, bourquin1990, bourquin1992}.

The results in this paper complement the a posteriori analysis
developed in \cite{jakbenlar2010} where CMS was applied to an elliptic
model problem, and the a posteriori analysis in \cite{jaklar2011},
where the elliptic eigenvalue problem was considered. Similar
techniques are used and results obtained here as in the previous two
publications. The frequency response problem does however require an
explicit treatment due to the indefinite nature of the problem.  We
further present a new adaptive strategy suitable for frequency sweep
analysis.

Other work on error analysis for CMS class methods include results by
Bourquin who considered the elliptic eigenvalue problem and derived a
priori bounds for the error in eigenvalues and eigenmodes
\cite{bourquin1990, bourquin1992}; and results by Yang, Gao, Bai, Li,
Lee, Husbands, and Ng \cite{yang2005algebraic} for the automated
multilevel substructuring method \cite{bennighof2004}, who derived a
criterion for mode truncation; together with results by Elssel and
Voss \cite{elssel2007priori} who showed that the same criterion
guarantees control of the error in the smallest eigenvalue in the
reduced problem.

Previous work on frequency response analysis based on CMS include that
by Bennighof and Kaplan \cite{bennighoffrequency}, who developed an
iterative method in which the response is split into two components,
one component near resonance and one component representing the
remainder of the response. The near resonant component is captured
using approximate global eigenmodes, and the remainder of the response
using substructure modes and iteration. A similar method was also
proposed by Ko and Bai \cite{ko2008high}. Error estimates for these
methods have, to the author's knowledge, not yet been developed.

Research on duality based a posteriori error estimation and adaptive
refinement strategies in finite element modeling has been ongoing
since the 1990's. For a general introduction to the subject in context
of finite element analysis we point the reader to \cite{rannbang2003,
  giles2003adjoint, eriksson1995}, and the references therein. We also
refer to \cite{irimie2001residual,
  larson2001posteriori,oden2005posteriori,rannacher1998posteriori} for
results that we feel are especially relevant in context of structural
mechanics.

The remainder of this paper is organized as follows. In Section 2 we
provide some prerequisite material and present the frequency response
problem in linear elasticity; in Section 3 we give an account of the
Craig-Bampton CMS in a variational setting; in Section 4 we derive an
a posteriori error estimates for the error in the displacements in the
reduced model measured in the energy norm; in Section 5 we demonstrate
our results in several numerical examples; and in Section 6 we
summarize our findings.

\section{The Frequency Response Problem}
Let $\Omega$ be a bounded domain in $\mathbb R^d$, $d=2$, $3$, with
boundary $\partial \Omega=\Gamma_D \cup \Gamma_N$, where $\Gamma_D
\cap \Gamma_N = \emptyset$, let $\mathcal T$ be a subdivision of
$\Omega$ into, for instance, triangles $(d=2)$ or tetrahedra $(d=3)$,
and let $V^h$ be the space of continuous, piecewise $p$th order vector
polynomials on $\mathcal T$ defined by $V^h=\{\mbf v \in
[H^1(\Omega)]^d : \,\mbf v\lvert_{\Gamma_D} = \mbf 0, \,\mbf v\lvert_T
\in [\mathcal{P}_p(T)]^d, \, \forall T \in \mathcal T\}$, where
$\mathcal{P}_p(T)$ is the space of $p$th order polynomials on element
$T$. Let further $a(\cdot, \cdot)$ be the bounded, coercive bilinear
form on $V^h \times V^h$ defined by $ a(\mbf{v},\mbf{w}) = 2(\mu
\mbf{\varepsilon}(\mbf{v}) : \mbf{\varepsilon}(\mbf{w})) + (\kappa
\nabla \cdot \mbf{v}, \nabla \cdot \mbf{w})$, where $
\mbf{\varepsilon}(\mbf{v}) : \mbf{\varepsilon}(\mbf{w})=\sum_{i,j=1}^d
\varepsilon_{ij}(\mbf{v})\varepsilon_{ij}(\mbf{w})$, let $(\cdot,
\cdot)$ denote the $L^2$ inner product on $V^h \times V^h$, and let
$b(\cdot)$ be the bounded linear form on $V^h$ given by $b(\mbf
v)=(\mbf{f}, \mbf{v}) + (\mbf{g}_N,\mbf{v})_{\Gamma_N}$, where $\mbf
f$ is a body force, and $\mbf g_N$ is a traction force.

The finite element frequency response problem reads: find $\mbf U
\in V^h$ such that
\begin{equation}\label{eq:fem_problem}
  a(\mbf U, \mbf v) + \iota\omega(\mathcal D\mbf U, \mbf v) - \omega^2(\mbf
  U, \mbf v) = b(\mbf v), \quad \forall \mbf v \in V^h.
\end{equation}
Given a basis in $V^h$, the following matrix form of
\eqref{eq:fem_problem} is obtained:
\begin{align}
  \mathbf K \bar{\mathbf U} + \iota\omega \mathbf D \bar{\mathbf U} - \omega^2\mathbf M \bar{\mathbf U} = \mathbf b,
\end{align}
where $\mathbf K$ is the stiffness matrix, $\mathbf M$ is the mass
matrix, $\mathbf D$ is a damping matrix, assumed to be on the form
$\mathbf D = \alpha \mathbf K + \beta \mathbf M$, $\alpha \geq 0$,
$\beta \geq 0$, i.e. Rayleigh damping, and $\mathbf b$ is the load
vector. The vector of coefficients of $\mbf U$ is denoted by
$\bar{\mathbf U}$.

\section{Component Mode Synthesis}

Let $\mathcal{S}=\{\Omega_i\}_{i=1}^n$ be a partition of $\Omega$ into
$n$ connected subdomains $\Omega_i$, such that each $\Omega_i =
\cup_{K\in\mathcal K_i} K$, for some subset $\mathcal K_i \subset
\mathcal K$. Let the interface between the subdomains be denoted by
$\Gamma$. An $a$-orthogonal decomposition
\begin{align}
  V^h = \bigoplus_{i=1}^n V^h_i,
\end{align}
of $V^h$ associated with $\mathcal S$ and $\Gamma$ may be constructed
by letting $V^h_i=\{ \mbf v \in V^h: v\lvert_{\Omega \setminus
  \Omega_i}=0\}$, $i=1, \ldots, n$, and by letting
\begin{align}
  V^h_{0}=\{ \mathcal{E} \mbf \nu \in V^h : \mbf \nu \in V^h\lvert_{\Gamma} \},
\end{align}
where $V^h\lvert_{\Gamma}$ denotes the trace space of $V^h$ associated
with $\Gamma$, and $\mathcal{E} \mbf \nu \in V^h$ denotes the energy minimizing
extension of a function $\mbf \nu \in V^h\lvert_{\Gamma}$ to $\Omega$. That
is, $\mathcal E \mbf \nu$ is defined by the problem: find
$\mathcal{E}\mbf{\nu} \in V^h$, such that
\begin{align}\label{eq:extension}
	a(\mathcal{E}\mbf{\nu}, \mbf{v}) &= 0, \quad \forall \mbf{v}
        \in V^h_i, \quad i=1, \ldots, n,
        \\ \mathcal{E}\mbf{\nu}|_\Gamma&=\mbf{\nu}.
\end{align}

A basis in each subspace $V^h_i$, $i=0, \ldots, n$, assumed to be of
dimension $k_i$, is obtained from the discrete eigenvalue problems:
find $(\Lambda_i, \mbf Z_i) \in \mathbb R \times V^h_i$ for $i=0,
\ldots, n$, such that
\begin{alignat}{3}\label{eq:disc_eigenprob}
  a(\mbf Z_i, \mbf v) & = & \Lambda_i (\mbf Z_i, \mbf v), 
  & \quad \forall \mbf v \in V^h_i, \quad i = 0, \ldots, n.
\end{alignat}
A reduced subspace $V^{h, \mbf{m}} \subset V^h$, where $\mbf
m=(m_i)_{i=0}^n$ is a multi-index, may be defined by letting
\begin{align}
  V^{h,\mbf{m}}=\bigoplus_{i=0}^n V_i^{h, m_i},
\end{align}
where
\begin{align}
  V^h_i \supset V^{h, m_i}_i=\mathrm{span}\{\mbf Z_{i, j}\}_{j=1}^{m_i}, \quad i=0, \ldots, n.
\end{align}

\subsection{The Reduced Problem}

Introducing the subspace $V^{h, \mbf{m}}$ in the model we get the
following reduced problem: find $\mbf U^{\mbf{m}} \in V^{h,\mbf{m}}$
such that
\begin{equation}\label{eq:reduced_problem}
  a(\mbf U^{\mbf m}, \mbf v) + \iota\omega(\mathcal D\mbf U^{\mbf m}, \mbf v) - \omega^2(\mbf
  U^{\mbf m}, \mbf v) = b(\mbf v), \quad \forall \mbf v \in V^{h, \mbf m}.
\end{equation}
Collecting the coefficients of the reduced basis functions columnwise
in the matrix $\mathbf V^{\mbf m}$, the matrix form of
\eqref{eq:reduced_problem} reads
\begin{align}
  \mathbf K^{\mbf m} \bar{\mathbf U}^{\mbf m} + \iota\omega \mathbf D^{\mbf m} \bar{\mathbf U}^{\mbf m} - \omega^2\mathbf M^{\mbf m} \bar{\mathbf U}^{\mbf m} = \mathbf b^{\mbf m},
\end{align}
where 
\begin{align}
  \mathbf K^{\mbf m}&=(\mathbf V^{\mbf m})^T \mathbf K \mathbf V^{\mbf m},\\
  \mathbf D^{\mbf m}&=(\mathbf V^{\mbf m})^T \mathbf D \mathbf V^{\mbf m},\\
  \mathbf M^{\mbf m}&=(\mathbf V^{\mbf m})^T \mathbf M \mathbf V^{\mbf m},\\
  \mathbf b^{\mbf m}&=(\mathbf V^{\mbf m})^T \mathbf b.
\end{align}

Next, we turn to error estimation and derive a posteriori error
estimates for the discrete error $\mbf E = \mbf U - \mbf U^{\mbf m}$
in the reduced problem.

\section{A Posteriori Error Analysis}
\subsection{Preliminaries}
We begin by remarking that for the error $\mbf E$ holds the Galerkin
orthogonality property
\begin{align}\label{eq:galerkin_orthogonality}
  a(\mbf E, \mbf{v}) + \iota\omega(\mbf E, \mbf v) -\omega^2(\mbf E, \mbf v)=0, \quad \forall \mbf{v}
        \in V^{h, \mbf{m}},
\end{align}
obtained by subtracting \eqref{eq:reduced_problem} from the finite
element formulation \eqref{eq:fem_problem}.

We further state the following notation, first introduced in
\cite{jakbenlar2010}. Here and below $\mathcal{R}_i: V^h \rightarrow
V^h_i$, $i=0, \ldots, n$, denote Ritz projectors, that is, the
projector defined by
\begin{align}\label{eq:ritz_projection}
  a(\mbf w - \mathcal R_i\mbf w, \mbf v)=0, \quad \forall \mbf v \in V^h_i,
\end{align}
and $\mathcal{R}:V^h \rightarrow V^h$, denotes decomposition, such that
\begin{align}
  \mathcal{R} \mbf w=\sum_{i=0}^n \mathcal{R}_i \mbf w.
\end{align}
The operators $\mathcal{P}_{i}^{m_i}:V^h_i
\rightarrow V^{h, m_i}_i$, $i=0, \ldots, n$, denote series
expansion in $V^{h, m_i}_i$, such that
\begin{align}
\mathcal
P_i^{m_i} \mbf w=\sum_{j=1}^{m_i} (\mbf w, \mbf Z_{i,j})\mbf Z_{i,j},
\end{align}
and the operator $\mathcal{P}^{\mbf{m}}:V^h \rightarrow V^{h,
  \mbf{m}}$, similarly denotes expansion in the space $V^{h, \mbf{m}}$,
such that 
\begin{align}
  \mathcal{P}^{\mbf{m}} \mbf u=\sum_{i=0}^n \mathcal{P}_i^{m_i}
  \mathcal{R}_i \mbf u.
\end{align}
We further let $\mbf{R}^h_i(\mbf{w}) \in
V^h_i$, $i=0, \ldots, n$, denote the discrete subspace residuals defined for $\mbf w \in V^h$ by
\begin{align}\label{eq:discrete_subspace_residual}
  (\mbf{R}^h_i(\mbf{w}), \mbf{v})&=b(\mbf{v}) - a(\mbf w, \mbf{v}) - \iota\omega(\mathcal D\mbf w, \mbf v) + \omega^2(\mbf w, \mbf v),
  \quad \forall \mbf{v} \in V^h_i.
\end{align}
We note here that the discrete residual $\mbf R^h(\mbf w) \in V^h$ is
defined by
\begin{align}\label{eq:discrete_subspace_residual2}
  (\mbf{R}^h(\mbf{w}), \mbf{v})&=b(\mbf{v}) - a(\mbf w, \mbf{v}) - \iota\omega(\mathcal D\mbf w, \mbf v) + \omega^2(\mbf w, \mbf v),
  \quad \forall \mbf{v} \in V^h,
\end{align}
and that the $L^2$ projection $\mathcal P_i\mbf R^h(\mbf w)$ of $\mbf R^h(\mbf w)$ onto $V^h_i$ is given by
\begin{align}
  (\mbf{R}^h(\mbf{w}) - \mathcal P_i\mbf R^h(\mbf w), \mbf{v})&=\quad \forall \mbf{v} \in V^h_i.
\end{align}
Hence, $\mbf{R}^h_i(\mbf w)=\mathcal P_i\mbf R^h(\mbf w)$, and
\begin{align}\label{eq:residual_relationship}
  (\mbf R^h_i(\mbf w), \mbf v)=(\mbf R^h(\mbf w), \mbf v), \quad \forall \mbf v \in V^h_i.
\end{align}
Finally, for $\alpha\geq0$ we define the operators
$\mathcal{L}_i^{\alpha}:V^h_i \rightarrow V^h_i$,
$i=0, \ldots, n$, by
\begin{align}\label{eq:discrete_pde_operator}
	(\mathcal{L}_i^{\alpha} \mbf{u},
  \mbf{v})&=\sum_{j=1}^{n_i}\Lambda_{i,j}^\alpha(\mbf{u},
  \mbf{Z}_{i,j})(\mbf{Z}_{i,j}, \mbf{v}), \quad \forall \mbf{u},
  \mbf{v} \in V^h_i, \quad i=0, \ldots, n,
\end{align}
where the $\mbf{Z}_{i,j} \in V^h_i$, $i=0, \ldots, n$, are given by
\eqref{eq:disc_eigenprob}.  We summarize some properties of
$\mathcal{L}_i^\alpha$ in the following lemma whose proof is straight
forward, and may be found in e.g \cite{jakbenlar2010}.
\begin{lma}
For $\mathcal L_i^\alpha$, $i=0, \ldots, n$, as defined in
\eqref{eq:discrete_pde_operator}, the following properties hold:
\begin{align}
  (\mathcal{L}_i\mbf{u}, \mbf{v})&=a(\mbf{u}, \mbf{v}), &\quad \forall
  \mbf{u}, \mbf{v} \in V^h_i, &\quad i=0, \ldots,
  n,\label{eq:Lproperties.1}\\ 
  \|\mathcal{L}_i^{1/2}\mbf{u}\|^2 &=\tn \mbf{u} \tn^2, &\quad
  \forall \mbf{u} \in V^h_i, &\quad i=0, \ldots,
  n, \label{eq:Lproperties.2} \\ 
  \|(I-\mathcal{P}_i^{m_i})\mbf{u}\|
  &\leq \frac{1}{\Lambda_{i, m_i+1}^\alpha}\|\mathcal{L}_i^\alpha
  \mbf{u}\|, &\quad \forall \mbf{u} \in V^h_i, &\quad i=0,
  \ldots, n. \label{eq:approximation_property}
\end{align}
\end{lma}

\subsection{A Posteriori Estimate in a Quantity of Interest}
First we show a straight forward a posteriori error estimate for
the discrete error in a quantity of interest defined by a linear
functional. The quantity of interest may for instance be the mean
stress on a part of the boundary, or the displacements near a point of
interest.

Let therefore $H(\cdot)=(\cdot, \mbf \psi)$, $\mbf \psi \in V^h$, be a
linear functional on $V^h$, and let the goal of solving
\eqref{eq:reduced_problem} be to accurately approximate $H(\mbf U)$.
We introduce the dual problem: find $\mbf{\Phi} \in V^h$, such that
\begin{align}\label{eq:discrete_dual_problem}
  a(\mbf{v}, \mbf{\Phi}) + \iota\omega(\mbf v, \mathcal D\mbf \Phi)-\omega^2(\mbf{v}, \mbf{\Phi})=H(\mbf{v}), \quad \forall \mbf{v} \in V^h.
\end{align}
Choosing $\mbf
v=\mbf E$ in \eqref{eq:discrete_dual_problem}, the error $H(\mbf E)$
may then be written
\begin{align}
  H(\mbf E)&=a(\mbf E, \mathcal R\mbf \Phi) + \iota\omega(\mathcal D\mbf E, \mathcal R\mbf \Phi) - \omega^2(\mbf E,
  \mathcal R\mbf \Phi) \label{eq:error_representation} \\
  &= \sum_{i=0}^n (\mbf R^h_i(\mbf U^{\mbf m}), \mathcal R_i\mbf \Phi) \\
  &= \sum_{i=0}^n (\mbf R^h(\mbf U^{\mbf m}), \mathcal R_i\mbf \Phi),
\end{align}
using \eqref{eq:residual_relationship}, and using the triangle inequality, the estimate
\begin{align}
  \lvert H(\mbf E) \rvert &\leq \sum_{i=0}^n \lvert (\mbf R^h(\mbf
  U^{\mbf m}), \mathcal R_i \mbf \Phi)
  \rvert, \label{eq:general_estimate}
\end{align}
immediately follows.

Now, using that $\{\mbf Z_{i,j}\}_{j=1}^{k_i}$ is an orthonormal basis
in $V^h_i$, we have for the residual $\mbf R^h_i(\mbf U^{\mbf m})$ and
the Ritz projection $\mathcal R_i \mbf \Phi$, respectively
\begin{align}
  \mbf R^h_i(\mbf U^{\mbf m}) &= \sum_{j=m_i+1}^{k_i} (\mbf R^h(\mbf U^{\mbf m}), \mbf Z_{i,j})\mbf Z_{i,j}, \\
  \mathcal R_i \mbf \Phi &= \sum_{j=1}^{k_i} \frac{a(\mbf \Phi, \mbf Z_{i,j})}{\lambda_{i,j}}\mbf Z_{i,j}.
\end{align}
Thus,
\begin{align}
  \lvert H(\mbf E) \rvert &\leq \sum_{i=0}^n\sum_{j=m_i+1}^{k_i} \frac{\lvert a(\mbf \Phi, \mbf Z_{i,j})(\mbf R^h(\mbf U^{\mbf m}), \mbf Z_{i.j}) \rvert}{\lambda_{i,j}}.
\end{align}
\begin{rem}
  The accuracy in the reduced model of course depends on the dimensions
  $m_i$ of the subspaces $V^{h, m_i}_i$, $i=0, \ldots, n$.  Looking at
  the estimate \eqref{eq:general_estimate}, each term
  \begin{align}\label{eq:indicators1}
    \eta_{J, i}= |(\mbf R^h(\mbf U^{\mbf m}), \mathcal
    R_i \mbf \Phi)|, \quad i=0, \ldots, n,
  \end{align}
  accounts for the contribution to the error $H(\mbf E)$ caused by
  truncating the basis in $V^{h, m_i}_i$, $i=0, \ldots, n$. The
  $\eta_{J, i}$ may then be used as a decision basis in an adaptive
  algorithm that automatically refines the subspaces contributing the
  most to the error $H(\mbf E)$.
\end{rem}
\begin{rem}
  To obtain $\eta_{J,i}$, we need the residual $\mbf R^h(\mbf U^{\mbf
    m})$ and the dual solution $\mbf \Phi$. The coefficient vector $\bar{\mathbf R}$ of the residual $\mbf
  R^h(\mbf U^{\mbf m})$ is given by the equation
  \begin{align}
    \mathbf M \bar{\mathbf R}=\mathbf b - \mathbf K \bar{\mathbf U}^{\mbf m} - \iota \mathbf D \bar{\mathbf U}^{\mbf m} + \omega^2 \mathbf M \bar{\mathbf U}^{\mbf m},
  \end{align}
  and the dual problem on matrix form reads
  \begin{align}
    \mathbf K \bar{\mathbf \Phi} + \iota \mathbf D \bar{\mathbf \Phi} - \omega^2 \mathbf M \bar{\mathbf \Phi}=\mathbf H,
  \end{align}
  The coefficient vector $\bar{ \mbf \Phi}_i$ of $\mathcal R_i \mbf
  \Phi$ is given by the equation
  \begin{align}
    \mathbf V_i^T \mathbf K \mathbf V_i \bar{\mbf \Phi}_i = \mathbf
    V_i^T \mathbf K \bar{\mbf \Phi}.
  \end{align}
  where $\mathbf V_i$ is a matrix containing the coefficients of a
  basis in $V^h_i$ in its columns. In the case of the modal basis
  $\{\mbf Z_{i,j}\}_{j=1}^{k_i}$, we then have $\mathbf V_i^T \mathbf
  K \mathbf V_i=\mathbf \Lambda_i$, where $\mathbf \Lambda_i$ is
  diagonal, and we obtain
  \begin{align}
     \eta_{J,i}&=\bar{\mathbf R}^T\mathbf M \mathbf V_i \bar{\mbf \Phi}_i \\
     &= \bar{\mathbf R}^T \mathbf M \mathbf V_i \mathbf \Lambda_i^{-1}\mathbf V_i^T \mathbf K \bar{\mbf \Phi} \\
    &=  (\mathbf b - \mathbf K \mathbf U^{\mbf m} - \iota \mathbf D \mathbf U^{\mbf m} + \omega^2 \mathbf M \mathbf U^{\mbf m})^T \mathbf V_i \mathbf \Lambda_i^{-1}\mathbf V_i^T \mathbf K \bar{\mbf \Phi} \\
     &=  (\mathbf b - \mathbf K \mathbf U^{\mbf m} - \iota \mathbf D \mathbf U^{\mbf m} + \omega^2 \mathbf M \mathbf U^{\mbf m})^T \mathbf W_i \hat{\mathbf \Lambda}_i^{-1}\mathbf W_i^T \mathbf K \bar{\mbf \Phi},
  \end{align}
  where we have used that $\mbf R^h(\mbf U^{\mbf m})$ is orthogonal to
  $V^{h, m_i}_i$ in the last equality, and introduced $\mathbf W_i$,
  which we assume contains the coefficients of the $k_i-m_i$
  eigenmodes $\{\mbf Z_{i,j}\}_{j=m_i+1}^{k_i}$, together with the diagonal matrix
  $\hat{\mathbf \Lambda}_i$ containing the corresponding eigenvalues.

In practice the dual problem is approximately solved, for instance
using a slightly larger reduced space $V^{h, \mbf d}$, where $\mbf
d=(d_i)_{i=0}^n$, and $m_i < d_i \leq k_i$, $i=0, \ldots, n$, compared
to what is used in the primal problem. Similarly, in the Ritz
projections of the dual solution onto the subspaces, approximations
may be used. Due to orthogonality it is then sufficient to project
onto the spaces $V^{h, d_i} \setminus V^{h,m_i}$, $i=0, \ldots, n$.
  
\end{rem}

\subsection{An Energy Norm Estimate}
The following a posteriori error estimate in
the energy norm $\tn \cdot \tn=\sqrt{a(\cdot, \cdot)}$ holds. 
\begin{thm}\label{thm:aposteriori_error_estimate}
Let $\mbf{U}$ and $\mbf{U}^{\mbf{m}}$ satisfy \eqref{eq:fem_problem}
and \eqref{eq:reduced_problem}, respectively. Then the
following a posteriori estimate holds for the energy norm of the
discrete error $\mbf{E} = \mbf{U} -
\mbf{U}^{\mbf{m}}$ in the approximation:
\begin{align}\label{thm:aposteriori_error_estimate.2}
  \tn \mbf{E} \tn &\leq \Bigl(\sqrt{I_1} + S(\omega)\sqrt{2I_2}\Bigr),
\end{align}
where $I_1=I_1(\mbf U^{\mbf m})$ and $I_2=I_2(\mbf U^{\mbf m})$ respectively, are defined by
\begin{align}
  I_1(\mbf U^{\mbf m}) &= \sum_{i=0}^n \frac{1}{\Lambda_{i,m_i+1}} \|  
  \mbf{R}^h_i (\mbf{U}^{\mbf{m}})\|^2, \\
  I_2(\mbf U^{\mbf m}) &= \sum_{i=0}^n \frac{1}{\Lambda_{i,m_i+1}^2} \| 
  \mbf{R}^h_i (\mbf{U}^{\mbf{m}})\|^2, \label{eq:I2}
\end{align}
and $S(\omega)$ is a stability factor, depending on the finite
element eigenvalues $\{\lambda_j^h\}_{j=1}^N$ and the frequency
$\omega$, defined by
\begin{align}
  S(\omega)=\sup_j \frac{\sqrt{(\omega^4 + \omega^2 c_j^2)\lambda^h_j}}{\sqrt{(\lambda_j^h - \omega)^2 + \omega^2c_j^2}},
\end{align}
where $c_j=\alpha \lambda^h_j + \beta$.

The terms $I_1$ and $I_2$ are given in matrix form by
  \begin{align}
    I_1(\bar{\mathbf U}^{\mbf m}) = \sum_{i=0}^N \frac{1}{\Lambda_{i, m_i+1}} {\bar{\mathbf R}_i^T \mathbf M_i \bar{\mathbf R}_i}, \\
    I_2(\bar{\mathbf U}^{\mbf m}) = \sum_{i=0}^N \frac{1}{\Lambda_{i, m_i+1}^2} {\bar{\mathbf R}_i^T \mathbf M_i \bar{\mathbf R}_i}.
  \end{align}
\end{thm}
\begin{proof}
As in \cite{jaklar2011} we split the dual solution $\mbf{\Phi}$ into
two parts $\mbf{\Phi}=\mbf{\Phi}_0 + \mbf{\Phi}_1$. Here
$\mbf{\Phi}_0$ satisfies
\begin{align}\label{eq:dual_split1}
  a(\mbf{v}, \mbf{\Phi}_0)=(\mbf{v}, \mbf{\psi}), \quad \forall \mbf{v} \in V^h,
\end{align}
and $\mbf{\Phi}_1$ satisfies
\begin{align}\label{eq:dual_split2}
  a(\mbf{v}, \mbf{\Phi}_1) + \iota\omega(\mathcal D\mbf v, \mbf \Phi_1) -\omega^2(\mbf v, \mbf \Phi_1)
  =\omega^2(\mbf{v}, \mbf{\Phi}_0) - \iota\omega(\mathcal D\mbf v, \mbf \Phi_0), \quad \forall \mbf{v} \in V^h.
\end{align}

Introducing this dual split in \eqref{eq:error_representation}, using
the Galerkin orthogonality property \eqref{eq:galerkin_orthogonality}
to subtract $\mathcal P \mbf \Phi$, we obtain
\begin{align}\label{eq:error_representation_split}
  H(\mbf E)&=a(\mbf E, \mathcal R \mbf \Phi_0 - \mathcal P  \mbf \Phi_0)
  + \iota\omega(\mathcal D\mbf E, \mathcal R\mbf \Phi_0 - \mathcal P \mbf \Phi_0) \\
  &\qquad - \omega^2(\mbf E,
  \mathcal R\mbf \Phi_0 - \mathcal P \mbf \Phi_0) \\
  &\qquad +a(\mbf E, \mathcal R\mbf \Phi_1 - \mathcal P \mbf \Phi_1) + \iota\omega(\mathcal D \mbf E, \mathcal R\mbf \Phi_1 - \mathcal P \mbf \Phi_1) \\
  &\qquad - \omega^2(\mbf E,
  \mathcal R \mbf \Phi_1 - \mathcal P \mbf \Phi_1) \\
  &= \sum_{i=0}^n (\mbf R^h_i(\mbf U^{\mbf m}), \mathcal R_i\mbf \Phi_0 - \mathcal P_i \mathcal R_i \mbf \Phi_0)\\
  &\qquad + \sum_{i=0}^n (\mbf R^h_i(\mbf U^{\mbf m}), \mathcal R_i\mbf \Phi_1 - \mathcal P_i \mathcal R_i \mbf \Phi_1) \\
  &\leq \sum_{i=0}^n \| \mbf R^h_i(\mbf U^{\mbf m})\| \| \mathcal R_i\mbf \Phi_0 - \mathcal P_i \mathcal R_i \mbf \Phi_0 \|\\
  &\qquad + \sum_{i=0}^n  \| \mbf R^h_i(\mbf U^{\mbf m})  \|  \| \mathcal R_i\mbf \Phi_1 - \mathcal P_i \mathcal R_i \mbf \Phi_1 \|.
\end{align}

  We choose $\mbf \psi=\mathcal{L}\mbf E$ in the dual problem \eqref{eq:discrete_dual_problem}, which gives $(\mbf E, \mbf \psi)=\tn \mbf E \tn^2$. 
  Using property \eqref{eq:approximation_property} in Lemma 1 together with the Cauchy-Schwarz inequality on the sums in \eqref{eq:error_representation_split}, we then have
  \begin{align}
    \tn \mbf E \tn^2 &\leq \Biggl ( \sum_{i=0}^n \frac{1}{\Lambda_{i,m_i+1}} \| \mbf{R}^h_i (\mbf{U}^{\mbf{m}})\|^2 \Biggl)^{1/2}  \Biggl (\sum_{i=0}^n  \tn \mathcal{R}_i \mbf{\Phi}_0 \tn^2 \Biggl )^{1/2} \\
    &\qquad + \Biggl ( \sum_{i=0}^n \frac{1}{\Lambda_{i,m_i+1}^2} \| \mbf{R}^h_i (\mbf{U}^{\mbf{m}})\|^2 \Biggl)^{1/2}  \Biggl (\sum_{i=0}^n  \| \mathcal{L}_i \mathcal{R}_i \mbf{\Phi}_1 \|^2 \Biggl )^{1/2}. \notag
  \end{align}
  Using that the decomposition of $V^h=\bigoplus_{i=0}^n V^h_i$ is
  $a$-orthogonal, we have
  \begin{align}
    \sum_{i=0}^n \tn \mathcal{R}_i \mbf \Phi_0 \tn^2 = \tn \mbf \Phi_0 \tn^2,
  \end{align}
  from Parseval's identity.
  It furthermore holds that
  \begin{align}
    \sum_{i=0}^n \| \mathcal{L}_i \mathcal{R}_i \mbf \Phi_1 \|^2
    &\leq 2 \|\mathcal{L} \mbf \Phi_1 \|^2,
  \end{align}
  cf. e.g. \cite{jakbenlar2010}.
  
  Combining the above results, we arrive at
  \begin{align}
    \tn \mbf{E} \tn^2 &\leq \biggl( \sum_{i=0}^n \frac{1}{\Lambda_{i,m_i+1}} \| \mbf{R}^h_i (\mbf{U}^{\mbf{m}})\|^2 \biggr)^{1/2} \tn \mbf \Phi_0 \tn \\
    &\qquad + \sqrt{2}\biggl( \sum_{i=0}^n \frac{1}{\Lambda_{i,m_i+1}^2} \| \mbf{R}^h_i (\mbf{U}^{\mbf{m}})\|^2 \biggr)^{1/2} 
    \| \mathcal L \mbf \Phi_1 \|.
  \end{align}
  We now turn to the question of stability of the dual solutions $\mbf \Phi_k$, $k=0, 1$.
  Beginning with $\mbf \Phi_0$ in \eqref{eq:dual_split1} we have
    \begin{align}
      a(\mbf v, \mbf \Phi_0 ) &= (\mbf v, \mathcal{L}\mbf{E}) \\
      &=  a(\mbf v, \mbf{E}), \quad \forall \mbf v \in V^h,
    \end{align}
    and hence $\tn \mbf \Phi_0 \tn = \tn \mbf E \tn$.  Next, the
    solution $\mbf \Phi_1$ to \eqref{eq:dual_split2} is given by the
    Fourier expansion
    \begin{align}\label{eq:dual_expansion}
      \mbf \Phi_1=  \sum_{j=1}^N \frac{(\omega^2 - \iota\omega c_j)(\mbf Z_j, \mbf \Phi_0)}{\lambda^h_j - \omega^2 + \iota\omega c_j}\mbf Z_j,
    \end{align}
    where $\mbf Z_j$, $j=1, \ldots, N$, is a basis of elastic
    eigenmodes in $V^h$, and $c_j=\alpha \lambda^h_j + \beta$.  Using
    that $(\mathcal L^\alpha \mbf Z, \mbf v)=\lambda^\alpha(\mbf Z,
    \mbf v)$, for $\mbf v \in V^h$, $\alpha \geq 0$, we then
    have
    \begin{align}
      \| \mathcal L \mbf \Phi_1 \|^2 
      &= \sum_j^N \frac{(\omega^2 -
        \iota\omega c_j)(\omega^2 + \iota\omega c_j)|(\mbf Z_j, \mathcal L
        \mbf \Phi_0)|^2} {(\lambda^h_j - \omega^2 + \iota\omega
        c_j)(\lambda^h_j-\omega^2-\iota\omega c_j)} \\ 
      &= \sum_j^N \frac{(\omega^4  + \omega^2 c_j^2)|(\mbf Z_j, \mathcal L \mbf \Phi_0)|^2} {(\lambda^h_j -
        \omega^2)^2 + \omega^2 c_j^2} \\ 
      &\leq \sup_j \frac{(\omega^4 + \omega^2 c_j^2)
        \lambda^h_j}{(\lambda^h_j - \omega^2)^2 + \omega^2 c_j^2} \tn
      \mbf \Phi_0 \tn^2,
    \end{align}
    and this completes the proof.
\end{proof}

\begin{rem}

In Theorem \ref{thm:aposteriori_error_estimate}, we may estimate further
using Young's inequality, to obtain the subspace indicators
\begin{align}\label{eq:indicators2}
  \eta_{a, i}&=\frac{2}{\Lambda_{i,m_{i+1}}}\|\mbf{R}^h_i(\mbf{U}^{\mbf{m}})\|^2 
  + \frac{4 S^2(\omega)}{\Lambda^2_{i,m_{i+1}}}\|\mbf{R}^h_i(\mbf{U}^{\mbf{m}})\|^2,
  \quad i=0, \ldots, n, 
\end{align}
which reflect to what degree modal truncation in each subspace influence
the energy norm of the error in the reduced solution.

Using these indicators we may design adaptive algorithms that
automatically determines suitable refinement levels in the individual
subspaces.  Such algorithms are outlined in Section
\ref{sec:numerical_examples}, along with the different numerical
examples.
\end{rem}
\begin{rem}
  The matrix form of equation \eqref{eq:discrete_subspace_residual2} for the subspace
  residual $\mbf R^h_i(\mbf U^{\mbf m})$ reads
  \begin{align}
    \mathbf V_i^T\mathbf M \mathbf V_i \bar{\mathbf R}_i=\mathbf V_i^T(\mathbf b - \mathbf K \bar{\mathbf U}^{\mbf m} - \iota \mathbf D \bar{\mathbf U}^{\mbf m} + \omega^2 \mathbf M \bar{\mathbf U}^{\mbf m}).
  \end{align}
  In the case of the modal basis, $\mathbf V_i^T \mathbf M \mathbf V_i
  = \mathbf I$, and
  \begin{align}
    \|\mbf R^h_i(\mbf U^{\mbf m})\|^2&=\bar{\mathbf R}_i^T\bar{\mathbf R}_i \\
    &=(\mathbf b - \mathbf K \bar{\mathbf U}^{\mbf m} - \iota \mathbf D \bar{\mathbf U}^{\mbf m} + \omega^2 \mathbf M \bar{\mathbf U}^{\mbf m})^T\mathbf V_i  \\
    &\qquad \times \mathbf V_i^T(\mathbf b - \mathbf K \bar{\mathbf U}^{\mbf m} - \iota \mathbf D \bar{\mathbf U}^{\mbf m} + \omega^2 \mathbf M \bar{\mathbf U}^{\mbf m}) \notag \\
    &=(\mathbf b - \mathbf K \bar{\mathbf U}^{\mbf m} - \iota \mathbf D \bar{\mathbf U}^{\mbf m} + \omega^2 \mathbf M \bar{\mathbf U}^{\mbf m})^T\mathbf W_i  \\
    &\qquad \times \mathbf W_i^T(\mathbf b - \mathbf K \bar{\mathbf U}^{\mbf m} - \iota \mathbf D \bar{\mathbf U}^{\mbf m} + \omega^2 \mathbf M \bar{\mathbf U}^{\mbf m}) \notag.
  \end{align}
  The subspace residuals may similarly to the Ritz projections of the
  dual solution be approximately computed. For instance using the
  subspaces $V^{h,d_i} \setminus V^{h,m_i}$, $i=0, \ldots, n$, or mass
  lumping techniques.
\end{rem}
\begin{rem}
In practical application we cannot evaluate $S(\omega)$ exactly as it
depends on the finite element eigenvalues which in general are unknown
and an approximation of $S(\omega)$ must be used. Such an
approximation may be obtained using a reduced model of the eigenvalue
problem.
\end{rem}

\section{Numerical Examples}\label{sec:numerical_examples}
In this section we apply the above developed theory in three numerical
examples. In the first example we describe and apply an adaptive
algorithm based on the estimate \eqref{eq:general_estimate} in a
single load case; in the second example we describe and apply an
algorithm based on \eqref{thm:aposteriori_error_estimate.2} in a
single load case: and in the third example we describe and apply an
adaptive algorithm based on \eqref{thm:aposteriori_error_estimate.2}
for the computation of a series of responses when $\omega$ varies over
a given range.

In all three examples we consider the domain $\Omega$ seen in Figure
\ref{fig:partition}. The domain is partitioned into subdomains
$\Omega_i$, $i=1, \ldots, 6$, interfacing at $\Gamma$. We assume that
the boundary is clamped at $x=0$ and stress free elsewhere. The
reference finite element model is piecewise linear, defined on a
triangular mesh containing approximately 7000 elements. The material
constants are $E=\rho=1$ and $\nu=0.29$, and the parameters $\alpha$
and $\beta$ in the Rayleigh damping are chosen as
$\alpha=\beta=0.025$.

\paragraph{Example 1.}
We consider the single load case $(\omega, \mbf f, \mbf g_N)$ where
$\omega=1$, $\mbf f=\mbf 0$, and $\mbf g_N=[0, -\exp(-100\lvert \mbf x-
  \mbf x_0 \rvert^2)]$, with $\mbf x_0=(0.7, 0.5)$. We assume that the
goal of the computation is to control the absolute error in the
functional $H(\cdot)=(\cdot, \pi \mbf \psi)$, where $\mbf
\psi=[\exp(-100\lvert \mbf x - \mbf x_1 \rvert^2), 0]$, and $\pi$ is
the nodal interpolation operator on $V^h$. This roughly amounts to
accurately computing the displacements in the $x$ direction near the
point $\mbf x_1=(0.95, 0.25)$. We use an adaptive algorithm of the
form outlined in Algorithm \ref{alg:adaptive_algorithm1}.
\begin{algorithm}[!t]
  \caption{}\label{alg:adaptive_algorithm1} 
  \begin{algorithmic}[1]
    \medskip
    \STATE Start with a guess of the subspace dimensions $\mbf m$.
    \STATE Solve the problem \eqref{eq:reduced_problem} for
    the displacements $\mbf U^{\mbf m}$.  \STATE Solve the dual
    problem \eqref{eq:discrete_dual_problem} for $\mbf \Phi$.  
    \STATE
     For each subspace $V^{h,m_i}_{i}$, compute the error indicators $\eta_{J,i}$ defined in
    \eqref{eq:indicators1}, and use them together with a refinement
    strategy, see Remark \ref{rem:refinement_rule}, to decide which subspaces are eligible for
    refinement and how much those subspaces should be refined. Refine those subspaces accordingly.
    \STATE Repeat steps 2--4
    until satisfactory results have been obtained.
  \end{algorithmic}
\end{algorithm}
\begin{rem}\label{rem:refinement_rule}
We use the following adaptive strategy: compute the
normalized subspace indicators $\tilde{\eta}_i$ by
$\tilde{\eta}_i=\eta_i/\sum \eta_i$. With the objective of adding a
maximum of NMODES eigenmodes in a maximum of NITS iterations, add
\begin{align}\label{eq:refinement_rule}
  l_i=\lfloor \tilde{\eta}_i \times \mathrm{NMODES}/\mathrm{NITS} \rfloor
\end{align}
modes in subspace $i=0,\ldots, n$, each iteration.
\end{rem}
Choosing the parameters $\mathrm{NMODES}=200$ and $\mathrm{NITS=10}$
in the adaptive strategy \eqref{eq:refinement_rule}, the objective of
the computation may be viewed as computing the output in $H(\cdot)$ as
accurately as possible using approximately 200 DOF distributed over
the course of 10 iterations. Using such an objective is motivated by
considerations regarding available precomputed basis functions and
computational resources. 

We start the algorithm with the subspace dimensions $m_i=1$. Each
iteration we compute an approximate dual solution $\hat{\mbf\Phi}$
using $m_i+10$ eigenmodes in each dual subspace basis.

In Table \ref{tab:subspace_dimensions1_approx} we see the subspace
dimensions evolving as the adaptive algorithm proceeds, and in Figures
\ref{fig:error1_approx}, we see the corresponding absolute error
$\lvert H(\mbf E) \rvert$ together with the estimated error
 as the adaptation proceeds.

For comparison we have also run the algorithm using an exact dual
solution $\mbf \Phi$. The resulting subspace dimensions are displayed
in Table \ref{tab:subspace_dimensions1_exact}, and the error and
estimate can be seen in Figure \ref{fig:error1_exact}. We see that the
adaptation is similar in both cases, and that the estimate with
approximate dual is accurate, although a slight underestimate is
introduced after the fourth iteration. The underestimation may be
alleviated by refining the dual more aggressively during adaptation.

\begin{figure}[!t]
  \begin{center}
    \scalebox{0.55}{\includegraphics{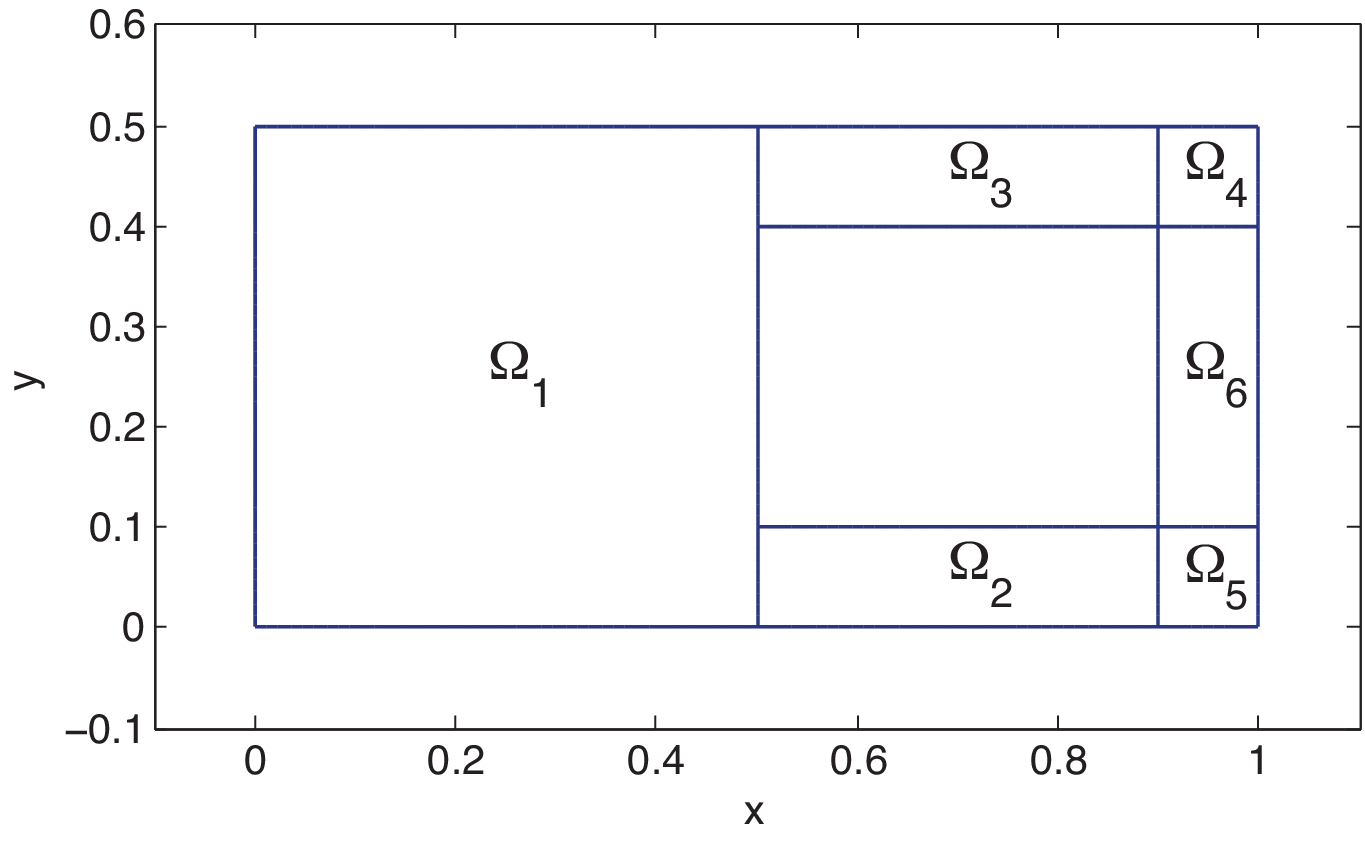}}
    \caption{The domain $\Omega$ partitioned into subdomains $\Omega_i$ interfacing at $\Gamma$.}\label{fig:partition}
  \end{center}
\end{figure}

\begin{figure}[!ht]
   \begin{center}
    \scalebox{0.55}{\includegraphics{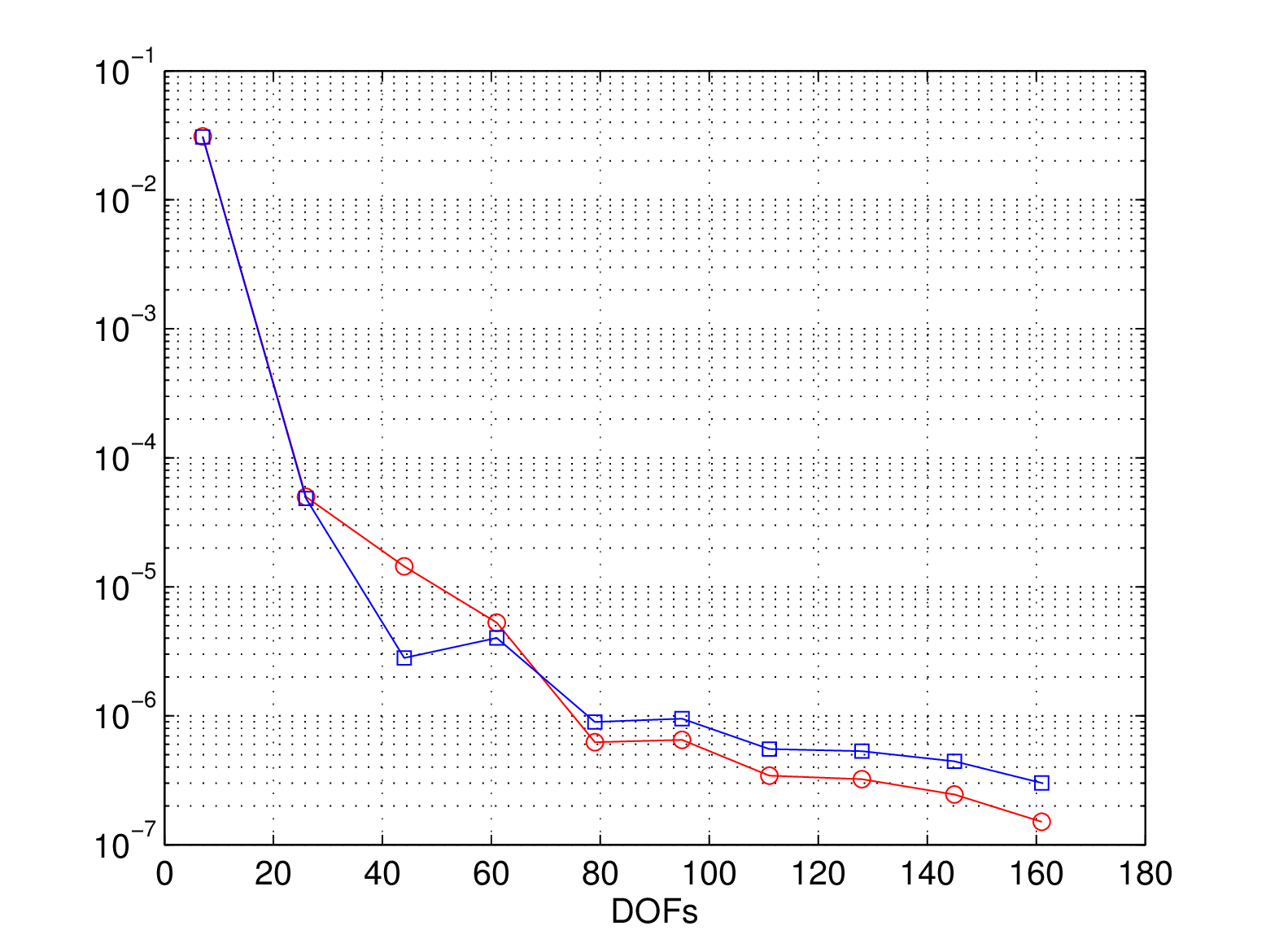}}
     \caption{Semilog plot of the absolute functional error and the
       estimate computed using an approximate dual solution vs. the
       number of DOFs as the adaptive algorithm
       proceeds in Example 1. Legend: square, error; circle, estimate.}\label{fig:error1_approx}
   \end{center}
\end{figure}
\begin{figure}[!ht]
   \begin{center}
     \scalebox{0.55}{\includegraphics{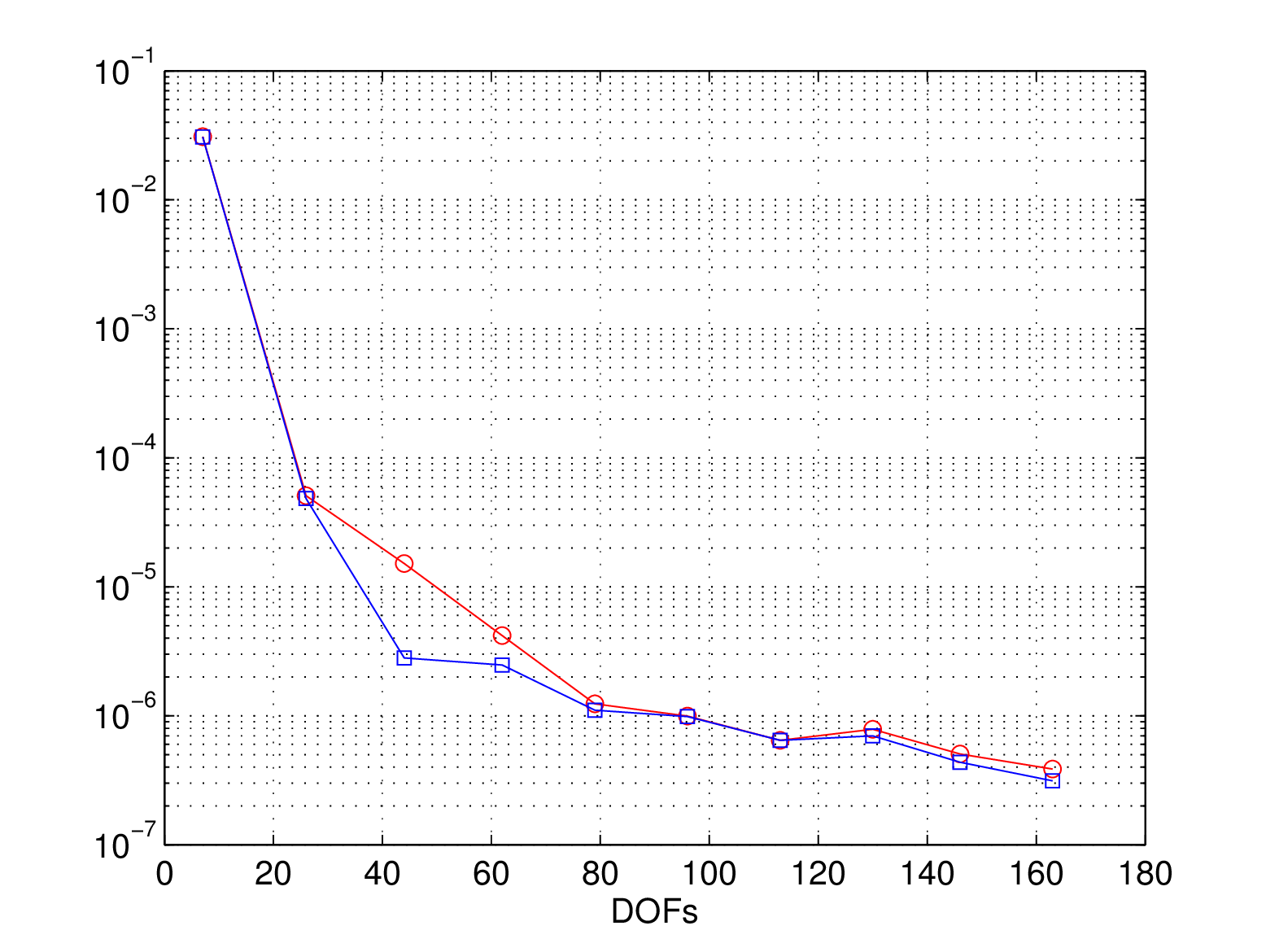}}
     \caption{Semilog plot of the absolute functional error and the
       estimate computed using the exact dual solution vs. the number
       of DOFs, as the adaptive algorithm proceeds in Example
       1. Legend: square, error;
       circle, estimate.}\label{fig:error1_exact}
   \end{center}
\end{figure}
\begin{table}[!ht]
  \centering
  \begin{tabular}{|c|c|c|c|c|c|c|c|}
    \hline
    iter. & $m_0$ & $m_1$ & $m_2$ & $m_3$ & $m_4$ & $m_5$& $m_6$ \\
    \hline
    \hline
    2  &  20  &    1  &    1  &    1  &    1  &    1  &    1  \\
    4  &  20  &    5  &   13  &   10  &    2  &    2  &    9  \\
    6  &  26  &   19  &   14  &   14  &    5  &    8  &    9  \\
    8  &  37  &   29  &   18  &   15  &    9  &   10  &   10  \\
    10 &  41  &   36  &   23  &   15  &   12  &   12  &   22  \\
    \hline
  \end{tabular}\caption{Iteration number and subspace dimensions 
    as the adaptation proceeds in Example 1 using an approximate dual
    solution.}\label{tab:subspace_dimensions1_approx}
\end{table}
\begin{table}[!ht]
  \centering
  \begin{tabular}{|c|c|c|c|c|c|c|c|}
    \hline
    iter. & $m_0$ & $m_1$ & $m_2$ & $m_3$ & $m_4$ & $m_5$& $m_6$ \\
    \hline
    \hline
    2  &  20  &   1   &   1   &   1   &   1   &   1   &   1 \\
    4  &  20  &   6   &  13   &  10   &   2   &   2   &   9 \\
    6  &  26  &  20   &  15   &  14   &   3   &   9   &   9 \\
    8  &  33  &  30   &  20   &  16   &   8   &   9   &  14 \\
    10 &  35  &  39   &  24   &  24   &   8   &  10   &  23 \\
    \hline
  \end{tabular}\caption{Iteration number and subspace dimensions  
    as the adaptation proceeds in Example 1 using an exact dual
    solution.}\label{tab:subspace_dimensions1_exact}
\end{table}
\paragraph{Example 2.}
Next, we consider the load case $(\omega, \mbf f, \mbf g_N)$, with
$\omega=\sqrt{3/2}$, $\mbf f=\mbf 0$, and $\mbf g_N=[\exp(-100\lvert
  \mbf x- \mbf x_0 \rvert^2), 0]$, with $\mbf x_0=(0.9, 0.25)$. In
this example we aim to control the energy norm of the error $\tn \mbf
E \tn$ as efficiently we can. We use an adaptive algorithm of the form
outlined in Algorithm \ref{alg:adaptive_algorithm2}.  Again we
use the adaptive strategy \eqref{eq:refinement_rule} with the
parameters $\mathrm{NMODES}=200$ and $\mathrm{NITS=10}$, and we start
the algorithm with subspace dimensions $m_i=1$. The stability factor
$S(\omega)$ is approximated using the eigenvalues from the
reduced eigenvalue problem: find $(\lambda^{\mbf m}, \mbf Z^{\mbf
  m})$, such that
\begin{align}
  a(\mbf Z^{\mbf m}, \mbf v)=\lambda^{\mbf m}(\mbf Z^{\mbf m}, \mbf
  v), \quad \forall \mbf v \in V^{h, \mbf
    m}.\label{eq:reduced_eigenvalue_problem}
\end{align}

We remark that although the size of the stability factor is not
crucial for the guiding of the adaptive algorithm in this example,
having a reasonable estimate is however important for quantitative
error estimation.

In Table \ref{tab:subspace_dimensions2_approx} we see the subspace
dimensions as the adaptation proceeds, and the corresponding error and
estimate can be seen in Figure \ref{fig:error2}. From the table
we see that the subspaces $V^{h, m_i}_i$ are refined symmetrically as
should be expected from the given load. We see in the figure that the
estimate provides an accurate bound on the error.

\begin{algorithm}[!t]
  \caption{}\label{alg:adaptive_algorithm2} 
  \begin{algorithmic}[1]
    \medskip
    \STATE Start with a guess of the subspace dimensions $\mbf m$.
    \STATE Solve the problem \eqref{eq:reduced_problem} for
    the displacements $\mbf U^{\mbf m}$.  
    \STATE  For each subspace $V^{h,m_i}_{i}$, compute the error
    indicators $\eta_{a,i}$ defined in \eqref{eq:indicators2},
    and use them together with a refinement
    strategy, see Remark
    \ref{rem:refinement_rule}, to decide which subspaces are eligible for
    refinement and how much those subspaces should be refined. Refine those subspaces accordingly.
    \STATE Repeat steps 2--3 until
    satisfactory results have been obtained.
  \end{algorithmic}
\end{algorithm}

\begin{figure}[!ht]
  \begin{center}
    \scalebox{0.55}{\includegraphics{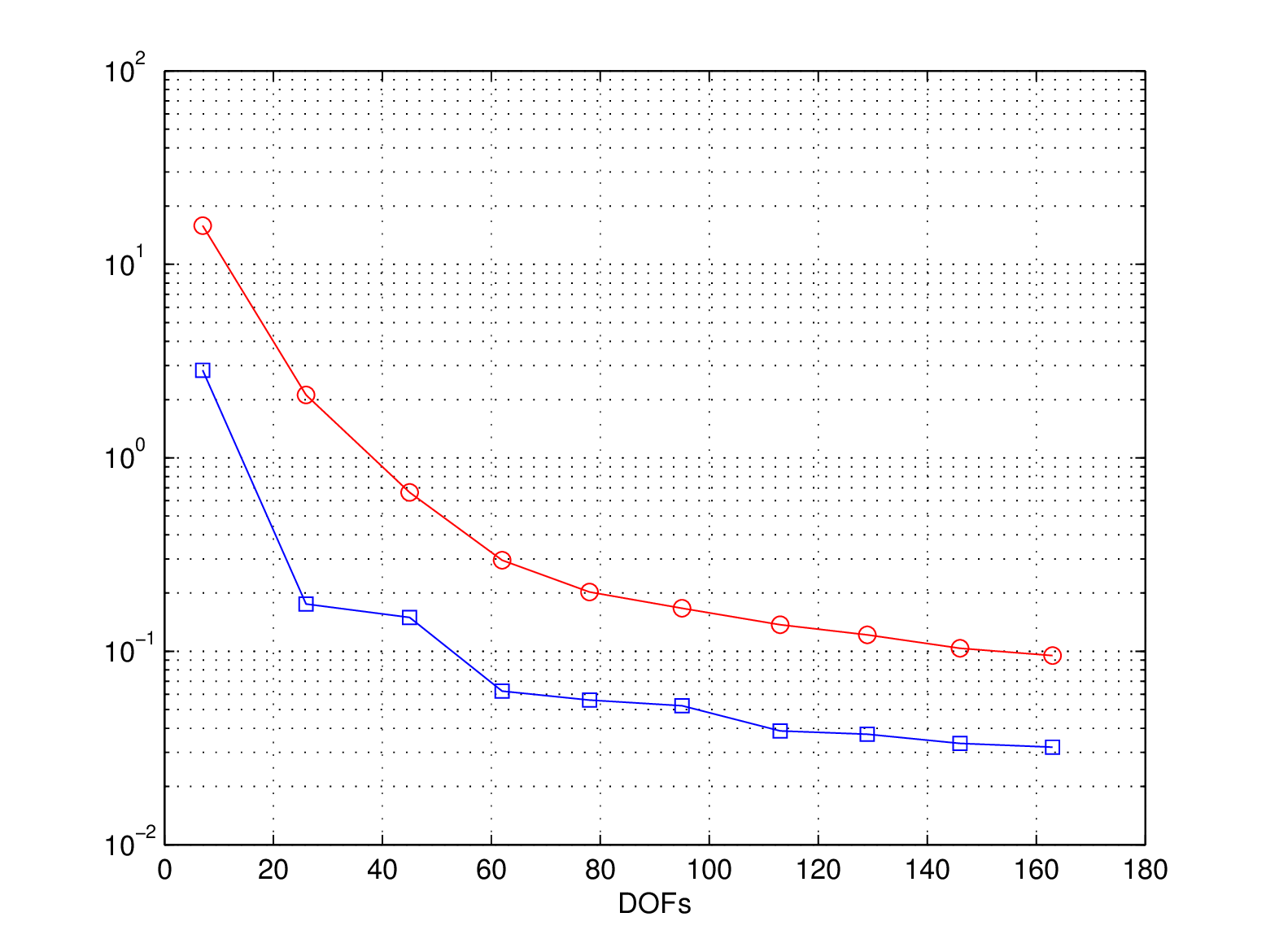}}
     \caption{Semilog plot of the energy norm of the error and the
       estimate computed using the exact stability factor vs. the
       number of DOFs, as the adaptive algorithm proceeds in Example
       2. Legend: square, error; circle, estimate.}\label{fig:error2}
  \end{center}
\end{figure}

\begin{table}[!ht]
  \centering
  \begin{tabular}{|c|c|c|c|c|c|c|c|}
    \hline
    iter. & $m_0$ & $m_1$ & $m_2$ & $m_3$ & $m_4$ & $m_5$& $m_6$ \\
    \hline
    \hline
    2  & 20  &    1   &   1   &   1  &    1   &   1  &    1  \\
    4  & 20  &    5   &   8   &   8  &    1   &   1  &   19  \\
    6  & 20  &   11   &   9   &   9  &    1   &   1  &   44  \\
    8  & 20  &   15   &  11   &  11  &    2   &   2  &   68  \\
    10 & 20  &   18   &  13   &  13  &    3   &   3  &   93  \\
    \hline
  \end{tabular}\caption{Iteration number and subspace dimensions 
    as the adaptation proceeds in Example
    2.}\label{tab:subspace_dimensions2_approx}
\end{table}

\paragraph{Example 3.}
In this example we're interested in computing the frequency response
for a set $\{(\omega_k, \mbf f_k, \mbf g_{N,k})\}$ of load cases. We
let the goal of the computation be to control the relative error
measured in energy norm $\tn \mbf E \tn/\tn \mbf U \tn$ for each load
case, and we assume that the following estimate
  \begin{align}
    \frac{\tn \mbf E \tn}{\tn \mbf U \tn} \lesssim \frac{1}{\tn \mbf U^{\mbf m} \tn}\Bigl(\sqrt{I_1} + S(\omega)\sqrt{2I_2}\Bigr),
  \end{align}
holds approximately.

An adaptive algorithm designed to handle this setting is outlined in
Algorithm \ref{alg:adaptive_algorithm3}. The algorithm utilizes that
if a basis has been adaptively constructed for a given $\omega \in
\mathbb R$, that basis is likely well suited for the case $\omega +
\varepsilon$ as well, when $\varepsilon$ is small and that the load
pattern has not changed significantly. The algorithm refines the
subspaces $V^{h, m_i}_{i}$ contributing to the error and coarsens
the subspaces that do not in order to keep the dimension of the
reduced subspace as small as it can.
\begin{algorithm}[H]
  \caption{}\label{alg:adaptive_algorithm3} 
  \begin{algorithmic}[1]
    \medskip
    \STATE For each load case $(\omega_k, \mbf f_k, \mbf g_{N,k})$ in a given set.
    \STATE Start with a guess of the subspace dimensions $\mbf m_k$.
    \STATE Solve the
    problem \eqref{eq:reduced_problem} for the displacements
    $\mbf U^{\mbf m_k}_k$.  
    \STATE For each subspace $V^{h,m_i}_{i}$, compute the error indicators
    $\eta_{a,i}$ defined in \eqref{eq:indicators2}, and use them
    together with a refinement strategy, see Remark
    \ref{rem:refinement_strategy3}, to decide which subspaces are
    eligible for refinement/coarsening and how much those subspaces
    should be refined/coarsened. Refine/coarsen those subspaces accordingly.
    \STATE Repeat steps 3--4 until satisfactory results have been
    obtained.  
    \STATE Let the resulting subspace dimensions be the
    starting guess for the next load case.
  \end{algorithmic}
\end{algorithm}
\begin{rem}\label{rem:refinement_strategy3}
  In Algorithm 3 we use a refinement strategy based on the following
  reasoning: begin by choosing a tolerance TOL, and let the objective
  for each load case be to refine the model such that the estimated
  relative error is the same as this tolerance, that is
  $\sqrt{\sum{\eta_{a,i}}}/\tn \mbf U^{\mbf m} \tn \approx
  \mathrm{TOL}$. Squaring both sides, we obtain
  \begin{align}
    \frac{\sum \eta_{a,i}}{\tn \mbf U^{\mbf m} \tn^2} \approx \mathrm{TOL}^2.
  \end{align}
  Assuming that each subspace should contribute equally to the
  error, so that
  \begin{align}
    \eta_{a,i} \approx \frac{1}{n}\sum \eta_{a,i},\label{eq:average_indicator}
  \end{align}
  we have that each indicator $\eta_{a,i}$ should fulfill
  \begin{align}
    \frac{\eta_{a,i}}{\tn \mbf U^{\mbf m} \tn^2} \approx \frac{\mathrm{TOL}^2}{n}.
  \end{align}
  By studying the difference
  \begin{align}
    \tau_{a,i}=\frac{\eta_{a,i}}{\tn \mbf U^{\mbf m} \tn^2}-\frac{\mathrm{TOL}^2}{n}, \quad i=0, \ldots, n,
  \end{align}
  we obtain a subspace indicator $\tau_{a,i}$, that is positive if
  refinement is required and negative if coarsening is required. By
  normalizing each indicator, we obtain a rough measure of how much
  each subspace should be refined or coarsened for
  \eqref{eq:average_indicator} to hold true.

  Hence, let $C=1/\sum \lvert \tau_{a,i} \rvert$, and choose $A_i, R_i
  \in \mathbb N \leq M_i \in \mathbb N$, where $M_i$ is the number of
  precomputed eigenmodes for the $i$th subspace. Then, if $\tau_{a,i}
  > 0$, add $\lfloor C\tau_{a,i} A_i \rfloor$ consecutive eigenmodes
  are to the $i$th basis subject to $\dim V^{h,m_i}_i \leq M_i$, and
  if $\tau_{a,i} < 0$, remove the last $\lfloor C\tau_{a,i} R_i
  \rfloor$ modes from the $i$th basis, subject to $V^{h,m_i}_i \geq
  1$.

  Further, if for some load case $(\omega_k, \mbf f_k, \mbf g_{N,k})$
  and subspace $V^{h,m_j}_j$ it holds that $\dim \mathcal V^{h,m_j}_j
  = M_j$, and $0 < \tau_{a,i} < \tau_{a,j}$, $i\neq j$, and
  $\sqrt{\sum{\eta_{a,i}}}/\tn \mbf U^{\mbf m} \tn > \mathrm{TOL}$, we
  consider the load case non resolvable given the current tolerance
  and maximum number of modes $M_i$, and we let the algorithm
  continues to the next load case.
\end{rem}

We let the load cases in the example be defined by $\omega^2_k=0.1,
0.2, \ldots, 10$, with $\mbf f$, and $\mbf g_N$ constant, chosen as in
Example 1. The parameters in the adaptive strategy outlined in Remark
\ref{rem:refinement_strategy3} are chosen as $\mathrm{TOL}=0.1$,
$A_i=M_i/10$, and $R_i=M_i/10$, where the number $M_i$ of precomputed
eigenmodes are $M_0=116$, $M_i=220$, $i=1, \ldots, 7$. As in Example
2, we solve the reduced eigenvalue problem
\eqref{eq:reduced_eigenvalue_problem} for a sufficiently large set of
eigenvalues needed to approximate the stability factor $S(\omega)$.

We start the algorithm with the subspace dimensions $m_i=1$. The
algorithm terminates after requiring a total of 59 refinement
iterations in order to satisfy the error tolerance in each of the 30
load cases.

In Figure \ref{fig:ex3_error_indicator} we have plotted the energy
norm of the solutions, relative errors, estimated relative errors, and
stability factors, for each of the computed load cases.  We see that
overall the error is estimated to a high degree of accuracy, and we
see that the estimate is close to the desired tolerance TOL=0.1, in
every load case. The stability factor is large near $\omega^2=2.0$,
due to proximity of eigenvalues at $\omega^2=1.6473$ and
$\omega^2=1.9812$. Since the norm of the solution increases when
approaching these values, the relative error decreases, however,
leading to less accuracy in the estimated error. This reflects the
fact that error estimation is more difficult near resonance
frequencies.

In Table \ref{tab:load_case_data} we have displayed the obtained
subspace dimensions and total dimension, the number of required
iterations, and the efficiency index $\mathrm{EI}=\sqrt{\sum \eta_{a,i}}/\tn
\mbf E \tn$, for each load case in the range $0.1 \leq \omega^2 \leq
3.0$. We see that typically only one or two iterations is required for
each load case, except in a few cases.
\begin{figure}[!t]
  \begin{center}
    \scalebox{0.55}{\includegraphics{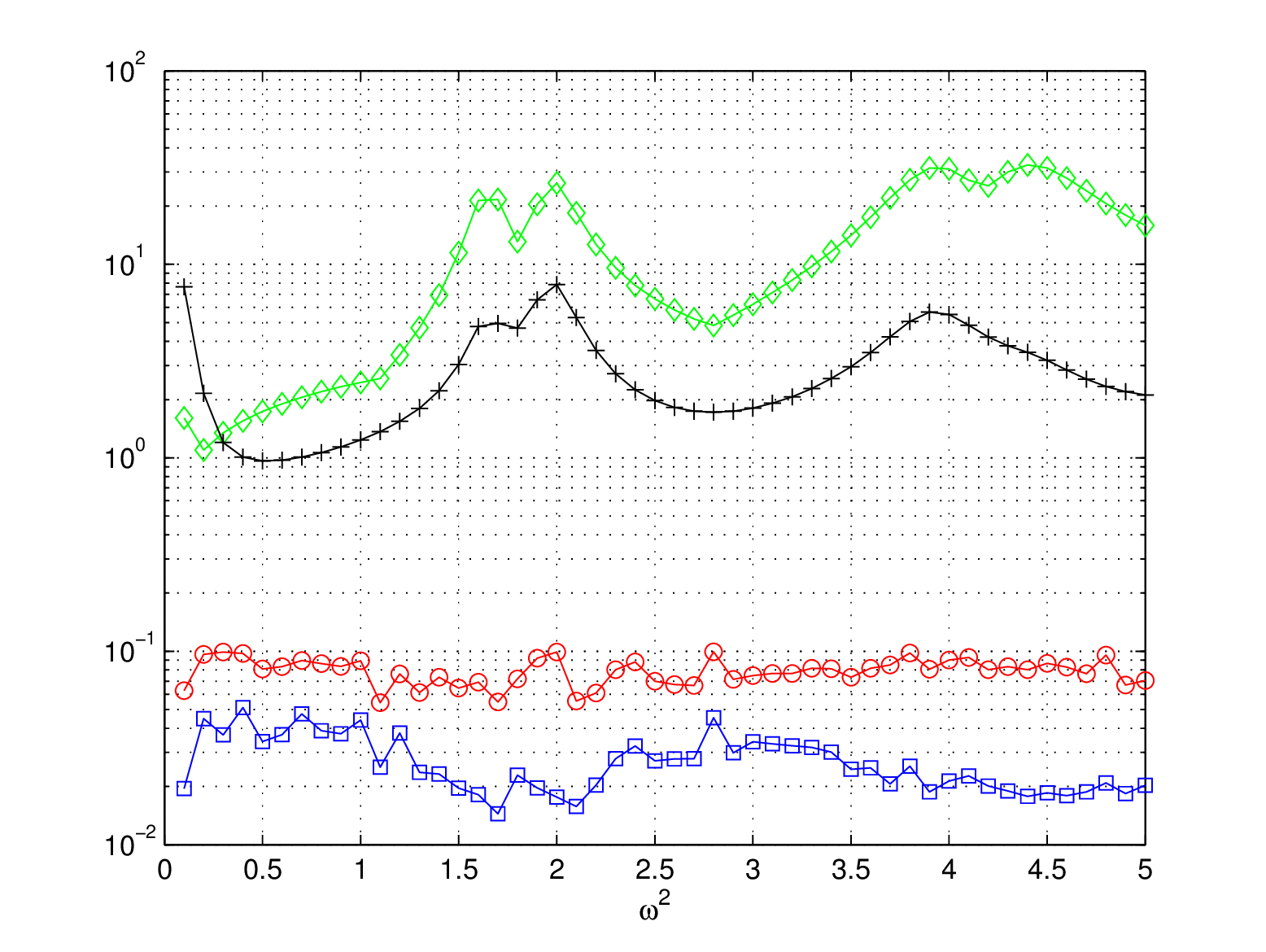}}
    \caption{Semilog plot of the energy norm of the
      solutions, the errors, estimated errors, and stability factors for each
      load case in Example 3. Legend: plus,
      energy norm; square, relative error; circle, estimated error; diamond, stability factor.
    }\label{fig:ex3_error_indicator}
  \end{center}
\end{figure}

\begin{table}[!ht]
  \centering
  \begin{tabular}{|c|c|c|c|c|c|c|c|c|c|c|}
    \hline
    $\omega^2$ & $m_0$ & $m_1$ & $m_2$ & $m_3$ & $m_4$ & $m_5$& $m_6$ & DOFs & its &  $\mathrm{EI}$ \\
    \hline
    \hline
    0.1  &  9   &   1   &   1  &    6   &   1   &   1   &   1 &   20  &  2  &  3.2035  \\
    0.2  &  9   &   1   &   1  &   39   &   1   &   1   &   2 &   54  &  3  &  2.0838  \\ 
    0.3  & 12   &   1   &   1  &   62   &   1   &   1   &   3 &   82  &  2  &  2.5641  \\
    0.4  & 14   &   1   &   1  &  120   &   1   &   1   &   2 &  108  &  5  &  2.3731  \\
    0.5  & 13   &   1   &   2  &  130   &   1   &   1   &   5 &  142  &  1  &  1.7461  \\
    0.6  & 15   &   1   &   1  &  157   &   1   &   1   &   3 &  154  &  3  &  2.2068  \\
    0.7  & 13   &   1   &   3  &  165   &   1   &   1   &   1 &  160  &  1  &  1.8104  \\
    0.8  & 12   &   1   &   1  &  172   &   1   &   1   &   5 &  173  &  1  &  1.8084  \\
    0.9  & 15   &   1   &   1  &  182   &   1   &   1   &   4 &  179  &  1  &  1.8141  \\
    1.0  & 13   &   1   &   4  &  187   &   1   &   1   &   1 &  188  &  1  &  2.2675  \\
    1.1  & 14   &   3   &   4  &  191   &   1   &   1   &   3 &  198  &  2  &  2.2362  \\
    1.2  & 11   &   9   &   5  &  193   &   1   &   1   &   4 &  206  &  2  &  2.0555  \\
    1.3  & 12   &   4   &   8  &  192   &   1   &   1   &   6 &  201  &  1  &  2.3509  \\
    1.4  & 14   &   6   &   2  &  190   &   1   &   1   &   5 &  213  &  2  &  2.7570  \\
    1.5  & 13   &  11   &   3  &  186   &   1   &   1   &  10 &  208  &  2  &  2.9929  \\
    1.6  & 14   &  15   &   4  &  178   &   1   &   1   &  14 &  223  &  4  &  3.7636  \\
    1.7  & 13   &  18   &  12  &  172   &   1   &   1   &   7 &  209  &  1  &  4.5314  \\
    1.8  & 11   &  18   &   8  &  168   &   1   &   1   &   3 &  203  &  1  &  3.1486  \\
    1.9  & 12   &  14   &   3  &  162   &   1   &   1   &  17 &  203  &  2  &  4.4424  \\
    2.0  & 10   &  13   &  13  &  159   &   1   &   1   &  14 &  205  &  2  &  4.3318  \\
    2.1  & 16   &  14   &  10  &  156   &   1   &   1   &  11 &  189  &  1  &  4.3043  \\
    2.2  & 13   &  13   &   5  &  152   &   1   &   1   &   6 &  181  &  1  &  3.0109  \\
    2.3  & 11   &  10   &  16  &  147   &   1   &   1   &   4 &  179  &  2  &  2.6907  \\ 
    2.4  & 11   &   9   &   7  &  143   &   1   &   1   &  13 &  175  &  1  &  2.5683  \\
    2.5  & 12   &   9   &  16  &  142   &   1   &   1   &   5 &  169  &  2  &  2.5769  \\
    2.6  & 15   &   9   &   8  &  142   &   1   &   1   &   7 &  174  &  2  &  2.4149  \\
    2.7  & 14   &   9   &  12  &  145   &   1   &   1   &   9 &  172  &  1  &  2.5297  \\
    2.8  & 11   &   9   &   7  &  147   &   1   &   1   &   4 &  175  &  1  &  2.3162  \\
    2.9  & 12   &   9   &  15  &  149   &   1   &   1   &   9 &  186  &  3  &  2.3814  \\
    3.0  & 19   &  10   &   7  &  151   &   1   &   1   &   3 &  176  &  1  &  2.5042  \\
    \hline                                                                          
  \end{tabular}\caption{Load case, resulting subspace dimensions,  total number of degrees of freedom, number of required iterations, and the efficiency index, for each load case in Example 3.}\label{tab:load_case_data}
\end{table}

\section{Summary and Outlook}
We have presented an a posteriori error analysis for reduced finite
element models of the frequency response problem in linear elasticity
constructed using component mode synthesis. We have derived estimates
for the error in the displacements measured in a linear goal
functional, as well as for the error measured in the energy norm. The
estimate reflects to what degree each CMS subspace influence the error
in the reduced solution allowing the design of adaptive algorithms
that automatically determines suitable subspace dimensions. We have
demonstrated our results in several numerical examples. The numerical
results follow the theoretical predictions to a high degree of
accuracy. The future of this research concerns its application in real
world three dimensional examples.
\section{Acknowledgments}
This research is supported by SKF and the Industrial Graduate School at
Ume{\aa} University.

\bibliography{references}
\end{document}